\documentclass{article}

\usepackage[all]{xypic}
\usepackage{amssymb, amsmath, amsthm}

\newcommand{\chapter}

\makeatletter

\newtheorem{thm}{Theorem}[section]
\newtheorem{cor}[thm]{Corollary}
\newtheorem{lem}[thm]{Lemma}
\newtheorem{prop}[thm]{Proposition}
\newtheorem{conj}[thm]{Conjecture}
\newtheorem{claim}[thm]{Claim}
\theoremstyle{definition}
\newtheorem{defn}[thm]{Definition}
\theoremstyle{remark}
\newtheorem{rem}[thm]{Remark}
\numberwithin{equation}{section}
\newtheorem{eg}[thm]{Example}

\newcommand{\set}[1]{\left\{#1\right\}}

\newcommand{\into}{\hookrightarrow}
\newcommand{\onto}{\twoheadrightarrow}
\newcommand{\id}{\mathbf{1}}
\newcommand{\quotes}[1]{\textquoteleft#1'}
\newcommand{\Ext}{\mbox{Ext}}
\newcommand{\RHom}{\mbox{RHom}}

\newcommand{\Hom}{\mbox{Hom}}
\newcommand{\End}{\mbox{End}}

\newcommand{\iso}{\stackrel{\sim}{\longrightarrow}}

\newcommand{\C}{\mathbb{C}}
\newcommand{\Z}{\mathbb{Z}}
\newcommand{\cO}{\mathcal{O}}
\renewcommand{\P}{\mathbb{P}}

\newcommand{\Perf}{\mbox{\textit{Perf}}}
\renewcommand{\bigwedge}{\mbox{\large{$\wedge$}}}
\renewcommand{\chapter}{%
    \secdef\Chapter\sChapter}%  typing \chapter{title} invokes
\newcommand{\Chapter}[2][?]{%
    \newpage%
    \refstepcounter{chapter}% update the counter
    \addcontentsline{toc}{chapter}% put something in the toc
        {\protect\numberline{\thechapter} #1}%
    {\raggedright%
        \normalfont\rmfamily\Huge\bfseries%\itshape%
        Chapter \thechapter:\par% write Chapter n
        }
    {\raggedright%
        \normalfont\rmfamily\Huge\bfseries%\itshape%
        #2\par}% write the chapter name
    \chaptermark{#1}% update the text which appears in the header
    \bigskip
    \noindent%
    \!\!}

\newcommand{\sChapter}[1]{
    \newpage%
    {\raggedright%
        \normalfont\rmfamily\Huge\bfseries%\itshape%
        #1\par}% write the chapter name
    \chaptermark{#1}% update the text which appears in the header
    \bigskip
    }

\renewcommand{\subsection}{%
    \secdef\Subsection\sSubsection}%   typing \subsection{title} invokes
\newcommand{\Subsection}[2][?]{%
    \refstepcounter{subsection}%
        \vspace{-0.1\baselineskip}%
    {\flushleft%
        \thesubsection\ \ %
        \normalfont\normalsize\scshape%
        #2%
        \par}%
    \addcontentsline{toc}{subsection}% put something in the toc
        {\protect\numberline{\thesubsection} #1}%
    \nopagebreak
    \vspace{0.5\baselineskip}%
    }

\newcommand{\sSubsection}[1]{%
    \vspace{-0.1\baselineskip}%
    {\flushleft%
        \normalfont\normalsize\sffamily#1%
        \par}%
        \nopagebreak
    \vspace{0.5\baselineskip}%
    }

\renewcommand{\subsubsection}{%
    \secdef\Subsubsection\sSubsubsection}%   typing \subsubsection{title} invokes
\newcommand{\Subsubsection}[2][?]{%
    \refstepcounter{subsubsection}%
    \vspace{-0.1\baselineskip}%
    {\flushleft%
        \thesubsubsection\ \ %
        \normalfont\normalsize\itshape%
        #2%
        \par}%
    \addcontentsline{toc}{subsection}% put something in the toc
        {\indent\protect\numberline{\thesubsubsection} #1}%
    \nopagebreak
    \vspace{0.5\baselineskip}%
    }

\newcommand{\sSubsubsection}[1]{%
    \vspace{-0.1\baselineskip}%
    {\flushleft%
        \normalfont\normalsize\sffamily#1%
        \par}%
        \nopagebreak
    \vspace{0.5\baselineskip}%
    }

\renewcommand{\section}{\@startsection
    {section}% the name.
    {1}% the level.
    {0mm}% the indent before the heading itself (minus means no indent).
    {-0.6\baselineskip}% the space before the heading,
    {0.5\baselineskip}% and the space after (negative is running header).
    {\normalfont\large\scshape\centering}}% the style.

\makeatother

\begin{document}

\title{Equivalences between GIT quotients of Landau-Ginzburg B-models}%
\author{Ed Segal}%

\maketitle

\begin{abstract}
We define the category of B-branes in a (not necessarily affine) Landau-Ginzburg B-model, incorporating the notion of R-charge. Our definition is a direct generalization of the category of perfect complexes. We then consider pairs of Landau-Ginzburg B-models that arise as different GIT quotients of a vector space by a one-dimensional torus, and show that for each such pair the two categories of B-branes are quasi-equivalent. In fact we produce a whole set of quasi-equivalences indexed by the integers, and show that the resulting auto-equivalences are all spherical twists.  
\end{abstract}

\tableofcontents
% ----------------------------------------------------------------
\section{Introduction}\label{sectintro}

The starting point for this paper is the celebrated result of Orlov \cite{orlov3} that the derived category of a Calabi-Yau hypersurface $Y$ in projective space is equivalent to the triangulated category of graded matrix factorizations for the homogeneous polynomial $f$ defining $Y$. In physicists' language this is the statement that the categories of topological B-branes are the same in the sigma model with target $Y$ and in the Landau-Ginzburg model with superpotential $f$. This is just part of a much deeper conjecture which goes back to Witten \cite{witten} (based on earlier observations by Vafa and others) and which states that the conformal field theories associated to these two models are different limit points in the same moduli space, the so-called Stringy K\"ahler Moduli Space (SKMS). The B-twist of a conformal field theory is expected to be independent of your position in the SKMS, so it follows that the B-twisted theories of each model should be equivalent and in particular that their B-brane categories are the same.

The basics of Witten's idea are easy to understand, even for a mathematician.
A \textit{Landau-Ginzburg model} is a K\"ahler manifold $X$ with a holomorphic function $W$ called the \textit{superpotential}. From such a thing one can write down a standard supersymmetric Lagrangian, analogous to the Lagrangian in classical mechanics coming from a Riemannian manifold equipped with a real-valued potential function. We consider the LG model
$$ X= \C^{n+1}_{x_1,...,x_n, p} \;\;\;\;\;\;\; W = f(x)p$$
where $f$ is a homogeneous degree $n$ polynomial in the $x_i$'s. Now we \quotes{gauge} this theory, which means we try and divide by the $\C^*$ symmetry under which each $x_i$ has weight 1 and $p$ has weight $-n$. The resulting theory should be the same, at least in some limit, as the theory coming from the quotient LG model. However we should take the quotient carefully, which means take a GIT quotient, and here we have two choices. One is the total space of the canonical bundle on $\mathbb{P}^{n-1}$
$$X_+ = K_{\mathbb{P}^{n-1}}$$
and the other is the affine orbifold
$$X_- = [\C^n / \Z_n]$$
The function $W$ descends to either of these, so they are both LG models. Physically, there is a parameter (the complexified Fayet-Iliopoulos parameter) in the Lagrangian of the gauged theory, and the theories on $X_+$ and $X_-$ are expected to appear at two different limits of this parameter. 

The important thing about a LG model is the set of critical points for the function $W$. On $X_+$ this will be the hypersurface $Y\subset \mathbb{P}^{n-1}$. If we assume that this is smooth, then locally around $Y$ the function $W$ is just quadratic in the normal directions, physically these directions are then \quotes{massive modes} and can be \quotes{integrated out} in the low-energy theory. What this means is that we should expect the theory coming from $(X_+, fp)$ to look essentially like a theory based just on $Y$, with no superpotential. 
We conclude that the theory on $Y$ is connected, in the moduli space of theories, to the theory on the orbifold $[\C^n/\Z_n]$ with superpotential $W$. This is called the \textit{Calabi-Yau/Landau-Ginzburg correspondence}. 

On the level of conformal field theories, this story is beyond the reach of current mathematical technology. However, all the spaces here are Calabi-Yau, which means the theories admit \quotes{B-twists} which are topological field theories (more precisely they are Topological Conformal Field Theories/Cohomological Field Theories) and these are much more tractable. And as mentioned above, the B-twisted theories should be independent of the FI parameter, and so the result to prove is that the B-model (the B-twisted TCFT) arising from $Y$ is the same as the B-model arising from the Landau-Ginzburg model $(X_-, fp)$.

The \textit{open sector} of a B-model is the category of B-branes, these are a type of boundary condition for the CFT. This should be a Calabi-Yau dg-category. For a LG model with $W=0$, the category of B-branes is (a dg-enhancement of) the bounded derived category of coherent sheaves. When $W\neq 0$ we need a generalization of this, the idea (due to Kontsevich) is to use objects that are like chain-complexes of sheaves, except that we now have $d^2=W$ instead of $d^2=0$. We define this category, denoted $Br(X,W)$ for a LG model $(X,W)$, in Section \ref{LGmodels}. In particular, the homotopy category of $Br(X_-,fp)$ is the category of matrix factorizations of $f$ on $X_-=[\C^n/Z_n]$, and the homotopy category of $(Y,0)$ is $D^b(Y)$.

Orlov's result is thus the statement that $H_0(Br(Y,0))=H_0(Br(X_-, fp))$, i.e. the homotopy categories of the open sectors of the two B-models are the same. In fact this goes a long way to proving that the whole of the B-models are the same, to get the full statement one would have to show that the open sectors are equivalent as Calabi-Yau dg-categories, the closed sectors should then follow by Costello's theorem \cite{costello2}. However this is not the aim of the current paper. Rather, we wanted to try to re-prove Orlov's result following more closely the ideas in Witten's construction, which does not appear in Orlov's proof. In particular we wanted to see the equivalence as the composition of two equivalences:
$$  (X_-, fp) \;\;\longleftrightarrow  \;\; (X_+, fp)  \;\;\longleftrightarrow \;\;Y$$
This both clarifies the result and suggests how to generalise it. For the first equivalence, we assume that the fundamental relationship between $X_-$ and $X_+$ is that they are birational, being related by a change of GIT quotient is just a special case of this. Hence we conjecture (Conjecture \ref{birational}) that the B-models associated to any two birational Calabi-Yau LG models are equivalent (to be more precise, we just conjecture that their open sectors are equivalent as dg-categories). This conjecture will be obvious to experts, it is a generalization of a theorem of Bridgeland \cite{bridgeland3} and lies in the same circle of ideas as Ruan's Crepant Resolution conjecture \cite{ruan}.

Unfortunately we do not get as far as addressing the second equivalence in this paper. We'll just remark that although it looks mysterious, it is just a global version of a fairly classical result by Kn\"orrer \cite{knorrer}, and a closely related result is proved by Orlov \cite{orlov2}.

What we actually manage to prove in this paper is a slight generalisation of the first equivalence, and so a small step towards our general conjecture. We show (Theorem \ref{GITequivalence}) that if $X_+$ and $X_-$ are two different GIT quotients of a vector space $V$ by $\C^*$, and $W$ is an invariant polynomial on $V$, then
$$Br(X_+,W)\simeq Br(X_-, W)$$
are equivalent dg-categories. Our proof borrows heavily from the work of Hori, Herbst and Page \cite{HHP}, in which they give a detailed physical argument for a generalisation of Orlov's result. Their key idea is a \textit{grade restriction rule}. Their reasoning involves A-branes and is mathematically rather mysterious, however the rule itself will be instantly familiar to anyone who knows Beilinson's Theorem \cite{beilinson}. Our improvement on their result, other than making it mathematically rigourous, is that they work only with the objects of the B-brane category whereas we include the morphisms as well (the massless open strings).

We want to explain one last aspect of the physics picture. The SKMS (the space in which the FI parameters live) can be explicitly described for our examples: it is a cylinder with one puncture, and the two GIT quotient LG models live at either end of the cylinder. The category of B-branes is the same for all points in this space, however we cannot trivialise it globally, i.e. there is monodromy. Therefore to get an equivalence between the B-brane categories of $X_+$ and $X_-$ we must pick a path between the two ends of the cylinder. Up to homotopy there are $\Z$ such paths, so we should find $\Z$ such equivalences. This is what we, and Orlov, find. By composing equivalences we get autoequivalences of the B-brane category at either end, this is the monodromy around the puncture. The puncture is the limit where the mass of a particular brane goes to zero, and the monodromy should be a Seidel-Thomas spherical twist around this brane. There is also monodromy around each end of the cylinder, this should just be given by tensoring with the $\cO(1)$ line-bundle.

Now we can explain the layout of the paper.

In Section \ref{sketch} we give a sketch of our method for the special case that $W=0$. This means we don't have to worry about LG models, we just deal with the more familiar case of derived categories of coherent sheaves. We show that we can construct $\Z$ many derived equivalences between $X_+$ and $X_-$, and that the resulting autoequivalences are spherical twists.

In Section \ref{LGmodels} we explain properly what a LG B-model is, and what the category of B-branes is. 

In Section \ref{quotients} we describe the class of examples we will consider. These are pairs of LG B-models $(X_+,W)$ and $(X_-,W)$ that are different GIT quotients of a vector space by $\C^*$.

In Section \ref{quasiequivalences} we prove our main result, that there are $\Z$ many quasi-equivalences between the categories of B-branes on $(X_+,W)$ and $(X_-, W)$.

In Section \ref{sphericaltwists} we describe the resulting auto-quasi-equivalences of the category of B-branes on $(X_+,W)$. We show (more-or-less) that they are spherical twists.

The technology of LG B-models is in its infancy, so many of the arguments of the last two sections are rather messy and ad-hoc. In particular the \quotes{more-or-less} of the previous paragraph is because we do not have a proper theory of Fourier-Mukai transforms. We apologise to the reader for this unsatisfactory state-of-affairs, and hope that later treatments will clean these results up a bit. 

\textbf{Acknowledgements.} I'd like to thank Richard Thomas for helpful suggestions, Manfred Herbst for patiently explaining \cite{HHP} to me, and the geometry department at Imperial College for sitting through some lectures on this material when it was in preliminary form.

Some results closely related to those of this paper (although using the \quotes{derived category of singularities} description of the category of B-branes) have been been found independently by \cite{HeW} and \cite[Section 7]{KKOY}.

\subsection{A Sketch Proof for $W=0$}\label{sketch}

As we will see in Section \ref{LGmodels}, a special case of the category of B-branes in a Landau-Ginzburg B-model is the category $\Perf(X)$ of perfect complexes on a smooth space $X$, which is a dg-model for the derived category $D^b(X)$. We thought it would be helpful to explain the proof of our results in this special case, as $D^b(X)$ is probably more familiar than $Br(X,W)$. Also the proof in this case is quite simple and still contains the important points for the more general case, the hard work in generalizing is mostly technicalities. 

For this sketch, we'll use the example of the standard three-fold flop. This is of course well understood and we will say nothing particularly original, but we will indicate afterwards how to generalise.

Let $V= \C^{4}$ with co-ordinates $x_1,x_2, y_1, y_2$, and let $\C^*$ act on $V$ with weight 1 on each $x_i$ and weight $-1$ on each $y_i$. There are two possible GIT quotients $X_+$ and $X_-$, depending on whether we choose a positive or negative character of $\C^*$. Both are isomorphic to the total space of the bundle $\cO(-1)^{\oplus 2}$ over $\mathbb{P}^1$.

Both are open substacks of the Artin quotient stack
$$\mathcal{X} = [V / \C^*]$$
given by the semi-stable locus for either character. Let
$$\iota_{\pm}: X_{\pm} \into \mathcal{X}$$
denote the inclusions. This stacky point of view makes it clear that there are (exact) restriction functors
$$\iota^*_{\pm}: D^b(\mathcal{X}) \to D^b(X_{\pm})$$
By $D^b(\mathcal{X})$ we mean the derived category of the category of $\C^*$-equivariant sheaves on $V$. This contains the obvious equivariant line-bundles $\cO(i)$ associated to the characters of $\C^*$. 

The unstable locus for the negative character is the set $\set{y_1=y_2=0}\subset V$. Consider the Koszul resolution of the associated sky-scraper sheaf:
$$K_- = \cO(2) \stackrel{(y_2, -y_1)}{\longrightarrow} \cO(1)^{\oplus 2} \stackrel{(y_1, y_2)}{\longrightarrow}\cO $$
Then $\iota_- K_-$ is exact, it is the pull-up of the Euler sequence from $\P^1_{y_1:y_2}$. On the other hand $\iota_+ K_-$ is a resolution of the sky scraper sheaf $\cO_{\P^1_{x_1:x_2}}$ along the zero section. Similar comments apply for the Koszul resolution $K_+$ of the set $\set{x_1=x_2=0}$.

Let
$$\mathcal{G}_t  \subset D^b(\mathcal{X})$$
be the triangulated subcategory generated by the line bundles $\cO(t)$ and $\cO(t+1)$.
This is the \textit{grade restriction rule} of \cite{HHP}, we are restricting to characters lying in the \quotes{window} $[t, t+1]$.

\begin{claim} For any $t\in\Z$, both $\iota^*_+$ and $\iota^*_-$ restrict to give equivalences
$$D^b(X_+) \stackrel{\sim}{\longleftarrow} \mathcal{G}_t \iso D^b(X_-)$$
\end{claim}
To see that these functors are fully-faithful it suffices to check what they do to the maps between the generating line-bundles, so we just need to check that
$$\Ext^\bullet_\mathcal{X}(\cO(t+k), \cO(t+l)) = \Ext^\bullet_{X_\pm}(\cO(t+k), \cO(t+l))$$
for $ k,l\in[0,1]$, i.e.
$$H^\bullet_\mathcal{X}(\cO(i)) =  H^\bullet_{X_\pm}(\cO(i))$$
for $i\in [-1,1]$, and this is easily verified. To see that they are essentially surjective we need to know that the the two given line bundles generate $D^b(X_\pm)$. This is essentially a corollary of Beilinson's Theorem \cite{beilinson}. One way to see it is to first observe that the set $\set{\cO(i), i\in \Z}$ generates $D^b(X_\pm)$  because $X_\pm$ is quasi-projective, then use twists of the exact sequence $\iota_\pm K_\pm$ repeatedly to resolve any $\cO(i)$ by a complex involving only $\cO(t)$ and $\cO(t+1)$.

So for any $t\in \Z$ we have a derived equivalence
$$\Phi_t: D^b(X_+) \iso D^b(X_-)$$
passing through $\mathcal{G}_t$. Composing these, we get auto-equivalences $$\Phi_{t+1}^{-1}\Phi_t: D^b(X_+) \iso D^b(X_+)$$
To see what these do, we only need to check them on the generating set of line-bundles $\set{\cO(t),\cO(t+1)}$. Applying $\Phi_t$ to this set is easy, it just sends them to the same line-bundles on $X_-$.\footnote{The easiest sign convention is to keep $y_1$ and $y_2$ as degree -1 on both sides, i.e. it's the $\cO(-1)$ bundle on $\P^1_{y_1:y_2}$ that has global sections. Otherwise $\Phi_t$ sends $\cO(t)$ to $\cO(-t)$.} To apply $\Phi_{t+1}^{-1}$ however, we first have to resolve $\cO(t)$ in terms of $\cO(t+1)$ and $\cO(t+2)$.  We do this using the exact sequence $\iota_- K_-(t)$. The result is that $\Phi_{t+1}^{-1}\Phi_t$ sends
\begin{eqnarray*}\cO(t) &\mapsto& [\cO(t+2) \stackrel{(-y_2, y_1)}{\longrightarrow} \cO(t+1)^{\oplus 2}]\\
\cO(t+1)& \mapsto &\cO(t+1) \end{eqnarray*}
\begin{claim} $\Phi_{t+1}^{-1}\Phi_t$ is an inverse spherical twist around $\cO_{\P^1_{x_1:x_2}}(t)$.\end{claim}
A spherical twist is an autoequivalence discovered by \cite{ST} associated to any \textit{spherical} object in the derived category, i.e. an object $S$ such that
$$\Ext(S,S) = \C \oplus \C[-n]$$
for some $n$ (i.e. the homology of the $n$-sphere). It sends any object $\mathcal{E}$ to the cone on the evaluation map
$$[\RHom(S, \mathcal{E})\otimes S \longrightarrow \mathcal{E}]$$
The inverse twist sends $\mathcal{E}$ to the cone on the dual evaluation map
$$[\mathcal{E} \longrightarrow \RHom(\mathcal{E}, S)^\vee \otimes S]$$
The object $\cO_{\P^1_{x_1:x_2}}(t) \simeq \iota_+K_-(t)$ is spherical, and the inverse twist around it sends $\cO(t+1)$ to itself and $\cO(t)$ to the cone
$$[\cO(t) \longrightarrow \iota_+K_-(t)] \;\;\;\simeq \;\;\; [\cO(t+2) \stackrel{(-y_2, y_1)}{\longrightarrow} \cO(t+1)^{\oplus 2}]$$
which agrees with $\Phi_{t+1}^{-1}\Phi_t$. To complete the proof of the claim we would just need to check that the two functors also agree on the Hom-sets between $\cO(t)$ and $\cO(t+1)$.

Now instead let $V =\C^{p+q}$ with co-ordinates $x_1,...,x_p,y_1,...,y_q$. Let $\C^*$ act linearly on $V$ with positive weights on each $x_i$ and negative weights on each $y_i$. The two GIT quotients $X_+$ and $X_-$ are both the total spaces of orbi-vector bundles over weighted projective spaces.

We must assume the Calabi-Yau condition that $\C^*$ acts through $SL(V)$. Let $d$ be the sum of the positive weights, so the sum of the negative weights is $-d$. The above argument goes through word-for-word, where now
$$\mathcal{G}_t = \left< \cO(t),..., \cO(t+d-1)\right>$$

\section{Landau-Ginzburg B-models}\label{LGmodels}

A Landau-Ginzburg model is a K\"ahler manifold $X$ equipped with a holomorphic function $W$. We are only interested in the B-model on $(X,W)$, and this doesn't need the metric, just the complex structure. Also we want to work in the algebraic world, so for us $X$ will be a smooth scheme (or stack) over $\C$.

When $W=0$, it is a standard slogan that the category of B-branes is the derived category $D^b(X)$ of coherent sheaves on $X$. However the category of B-branes should really be a dg-category, whose homotopy category is $D^b(X)$ (for background on dg-categories, we recommend \cite{toen}). A good model is given by $\Perf(X)$, the category of perfect complexes. The objects of $\Perf(X)$ are bounded complexes of finite-rank vector bundles, and the morphisms are given by
$$\Hom(E^\bullet, F^\bullet) = \Gamma(\mathcal{H}om(E^\bullet, F^\bullet)\otimes \mathcal{A}^{0,\bullet})$$
(this is what we might call the \quotes{Dolbeaut} version of $\Perf(X)$, other versions are possible as we will discuss below). The differential here is a sum of the Dolbeaut differential $\bar{\partial}$ and the differential on $\mathcal{H}om(E^\bullet, F^\bullet)$, which itself is the commutator with the differentials on $E^\bullet$ and $F^\bullet$. The homology of this complex is 
$$\Ext^\bullet(E^\bullet, F^\bullet)= \Hom_{D^b(X)}(E^\bullet, F^\bullet)$$
 Futhermore since $X$ is smooth every object in $D^b(X)$ is quasi-isomorphic to a complex in $\Perf(X)$, so $H_0(\Perf(X)= D^b(X))$ as required.

We need to generalise this for $W\neq 0$. Kontsevich's idea was to modify the definition of a chain-complex, replacing $d^2 = 0$ with $d^2 = W$. This doesn't make sense on a $\Z$-graded complex, so the usual procedure (at least in the mathematics literature) is to work instead with $\Z_2$-graded complexes. However there is another possibility, standard in the physics literature, which is to replace the \quotes{homological} grading with the notion of \textit{R-charge} (strictly speaking, \textit{vector R-charge}). This is a geometric action of $\C^*$ on $X$, under which $W$ must have weight 2. Then we can define a \textit{B-brane} to be a $\C^*$-equivariant vector bundle $E$, with an endormorphism $d$ of R-charge 1, and the condition $d^2 = W\id_E$ makes sense. 
If the $\C^*$ action is trivial then we are forced to take $W=0$, and we recover the definition of a perfect complex. Also, the definition of the morphism chain-complexes in $\Perf(X)$ adapts easily, as we shall see.

\begin{defn} A Landau-Ginzburg B-model is the following data:
\begin{itemize}
\item A smooth $n$-dimensional scheme (or stack) $X$ over $\C$.
\item A choice of function $W\in \mathcal{O}_X$ (the \quotes{superpotential}).
\item An action of $\C^*$ on $X$ (the  \quotes{vector R-charge}).
\end{itemize}
such that 
\begin{enumerate}
\item $W$ has weight (\quotes{R-charge}) equal to 2.
\item $-1\in \C^*$ acts trivially.
\end{enumerate}
\end{defn}

From now on we'll call the $\mathbb{C}^*$ acting in this definition $\mathbb{C}^*_R$
to distinguish it from other $\C^*$ actions that will appear later.

\begin{rem}
In physics terms, Axiom 2 follows from the fact that the axial R-charge symmetry is acting trivially. It implies that the sheaf of functions $\cO_X$ is supercommutative under the $\C^*_R$ grading. We could relax it, but  keep supercommutativity, by allowing $X$ to be a superspace. 
\end{rem}

\begin{defn}
A \textit{B-brane} on a Landau-Ginzburg B-model $(X,W)$ is a finite-rank vector bundle
$E$, equivariant with respect to $\C^*_R$, equipped with an endomorphism
$d_E$ of R-charge 1 such that $d_E^2 = W\cdot\id_E$.
\end{defn}

If we wanted to be more pretentious we could say that $X$ is a space endowed with a sheaf of curved algebras ($W$ is the curvature) and that a B-brane is a locally free sheaf of curved dg-modules over $X$.

We can shift the R-charge on a B-brane $E$ by tensoring with a line bundle
associated to a character of $\C^*_R$. We denote these shifts by $E[n]$ for
$n\in \Z$. This agrees with the homological shift functor in the following special case:

\begin{eg}Let $W=0$ and $\C^*_R$ act trivally. Then a B-brane is just a bounded
complex of vector bundles.
\end{eg}

Note that since $-1\in \C^*_R$ acts trivially every B-brane splits as a direct sum
$$E = E^{ev} \oplus E^{od}$$
of its $\Z_2$-eigen-bundles, and $d_E$ exchanges these sub-bundles. There is a weaker definition of Landau-Ginzburg B-model where we keep only the trivial action of $\Z_2\subset \C^*_R$, thus only this $\Z_2$-grading remains. We shall make no use of this weaker definition, except for the following example.

\begin{eg} Let $X=\C^n$ and $W$ be any polynomial. This defines a LG B-model in the weak ($\Z_2$-graded) sense. Then for a B-brane $(E, d_E)$ both $E^{ev}$ and $E^{od}$ must be trivial bundles, so $d_E$ is given by a matrix
$$d_E = \left(\begin{array}{cc} 0 & d_E^0 \\ d_E^1 & 0 \end{array}\right)$$
whose square is $W\id$. This is a called a matrix factorization of $W$.
\end{eg}

We can't in general add R-charge to this example. But we can if we orbifold it, as follows.

\begin{eg} \label{affineorbifold}
Let $X=[\C^{n}_{x_1,...,x_n}\times \C^*_p \;/\; \C^*_G]$, where $\C^*_G$ (the \textit{gauge group}) acts with weight 1 on each $x_i$ and weight $-k$ on $p$. This is equivalent as a stack to $[\C^n/\Z_k]$. Let $\C^*_R$ act with weight 0 on each $x_i$ and weight 2 on $p$. If we pick a superpotential $W= f(x)p$ where $f(x)$ is a homogeneous degree $k$ polynomial in the $x_i$'s, then this defines a LG B-model (for $k=n$ it is the orbifold phase of the Witten construction described in the introduction). Every $\C^*_R$-equivariant vector bundle on $X$ is the direct sum of $\C^*_R$-equivariant line-bundles, these are given by the lattice
$$\Z^2/(-k,2)$$
This bijects with the subset $\Z\times[0,1]\subset \Z^2$. This means that we can consider a B-brane on $(X,W)$ to be given by a pair $(E^0,E^1)$ of graded free modules over the ring $\C[x_1,...,x_n]$ where each $x_i$ has degree 1, and graded maps 
$$d_E^0:E^0 \to E^1,\;\;\;\;\;\; d_E^1: E^1 \to E^0$$
with $d_E^0d_E^1=d_E^1d_E^0=f$. This is called a graded matrix factorization.
\end{eg}

Now we want to define the morphisms between two B-branes. We will precisely mimic the construction of $\Perf$, by first defining a homomorphism bundle and then taking derived global sections of it.

Recall that a B-brane on the  LG B-model $(X,0)$ is a $\C^*_R$-equivariant bundle $E$ on $X$ equipped with an endomorphism $d_E$ of R-charge 1 whose square is zero. Let $\mathbf{dg_R Vect}(X)$ be the category whose objects are B-branes on $(X,0)$ and whose morphisms are all morphisms of vector bundles. This is a dg-category, and when the $\C^*_R$ action on $X$ is trivial it is just the category $\mathbf{dgVect}(X)$ of complexes of vector bundles on $X$. It is also a monoidal category, since we can tensor equivariant bundle and their endomorphisms in the usual way. 

Now let $(X,W)$ be any LG B-model, and let $(E, d_E)$, $(F, d_F)$ be two B-branes on $(X,W)$. We have a $\C^*_R$-equivariant vector bundle 
$$\mathcal{H}om(E,F) = E^\vee\otimes F$$
and this carries an endomorphism 
$$d_{E,F}= \id_E^\vee\otimes d_F - d_E^\vee\otimes \id_F $$
of R-charge 1. One can check that 
$$d_{E,F}^2 = 0$$
(the two copies of $W$ that appear cancel each other). This means that the pair $(\mathcal{H}om(E,F), d_{E,F})$ is an object of $\mathbf{dg_R Vect}(X)$. Furthermore, given a third B-brane $(G, d_G)$, we have composition maps
$$\mathcal{H}om(E,F)\otimes \mathcal{H}om(F,G) \to \mathcal{H}om(E,G)$$
and these are closed and of degree zero.

\begin{defn} Given an LG-model $(X,W)$ we define a category $\mathcal{B}r(X,W)$
enriched over the category $\mathbf{dg_R Vect}(X)$.
The objects of $\mathcal{B}r(X,W)$ are the B-branes on $(X,W)$, and the morphisms
between two branes $E$ and $F$ are given by 
$$(\mathcal{H}om(E,F),d_{E,F})$$
\end{defn}

We need to fix a monoidal functor $R\Gamma: \mathbf{Vect(X)^{\C^*_R}}\to \mathbf{dgVect^{\C^*_R}}$ that sends a $\C^*_R$-equivariant vector bundle to a bounded $\C^*_R$-equivariant chain-complex of vector spaces that computes its derived global sections. Since we are working with smooth spaces over $\C$ we will use Dolbeaut resolutions, i.e. we define
$$R\Gamma(E) = (\Gamma (E\otimes \mathcal{A}_X^{0,\bullet}), \bar{\partial})$$
 but we could also use other models such as \v{C}ech resolutions with respect to some $\C^*_R$-invariant open cover.

Now $\mathcal{H}om(E,F)$ is an object in $\mathbf{dg_R Vect}(X)$. This means that 
$$R\Gamma(\mathcal{H}om(E,F)) =  \Gamma(\mathcal{H}om(E,F)\otimes\mathcal{A}_X^{0,\bullet})$$
is a bi-complex, graded by R-charge and by Dolbeaut degree, with differential
$$d_{E,F} + \bar{\partial}$$
 As usual we may collapse this bi-complex to a complex. If we apply this to all pairs of branes simultaneously we get the following:
 
\begin{defn}
Given an LG-model $(X,W)$ we define the dg-category of B-branes to be
$$Br(X,W) := R\Gamma (\mathcal{B}r(X,W))$$
\end{defn}

The monoidalness of $R\Gamma$ ensures that this is indeed a category.

\begin{eg} \label{perf}Let $W=0$ and $\C^*_R$ act trivially on $X$. Then $Br(X,0)
= \Perf(X)$, the category of perfect complexes. Since $X$ is smooth the homotopy
category of this is
$$H_0(Br(X,0)) = D^b(X)$$
\end{eg}

\begin{eg} Let $X=[\mathbb{C}^n/\Z_n]$ as in example \ref{affineorbifold}. Then the functor $\Gamma$ means \quotes{take $\mathbb{Z}_n$-invariants}, and this is exact, so we may let $R\Gamma_X = \Gamma_X$. The homotopy
category of $Br(X,W)$ is the category of graded B-branes $DGrB(W)$ defined by Orlov \cite{orlov3}. 
\end{eg}

\begin{rem} $Br(X,W)$ should only depend on a (Zariski) neighbourhood of the critical locus of $W$. This has been proved (without R-charge and on the level of homotopy categories) by Orlov \cite{orlov}.
\end{rem}

\begin{rem} As far as we are aware this definition is new in the mathematics literature, but it is almost classical in the physics literature, see e.g. \cite{HerbLaz}.\end{rem}

\begin{rem} We could make the definition of a B-brane more general by allowing the endomorphism $d_E$ to be derived, i.e.
$$d_E \in \Gamma(\mathcal{E}nd(E)\otimes \mathcal{A}_X^{0,\bullet})$$
with R-charge plus Dolbeaut degree equal to 2. Similarly we could generalize the definition of LG B-model by allowing $W$ to be a closed element of $\mathcal{A}_X^{0,\bullet}$.
The advantage of this more general definition of B-brane is that the resulting category contains mapping cones, i.e. it is pre-triangulated. However notice that in Example \ref{perf} above $\Perf(X)$ is already pre-triangulated, this leads us to suspect that at least when $W\in \cO_X$ our more restricted category of B-branes is in fact pre-triangulated as well. When $X$ is affine this is obvious.\end{rem} 

\begin{rem} \label{spectralsequence}
Since the Hom sets are actually bi-complexes, and the Dolbeaut grading is bounded, we have a spectral sequence converging to the homology of $R\Gamma(\mathcal{H}om(E,F))$ whose first page is
$$(H^\bullet(\mathcal{H}om(E,F)), d_{E,F})$$ 
\end{rem}

A map $f: (X,W) \to (X', W')$ of LG B-models is just a map from $X$ to $X'$ commuting with the R-charges and such that $f^*W'=W$. Assuming that the derived global sections functors $R\Gamma_X$ and $R\Gamma_{X'}$ are chosen compatibly we get a dg-functor
$$f^*: Br(X', W') \to Br(X, W)$$

Similarly a \textit{birational map} between $(X,W)$ and $(X',W')$ is a birational map from $X$ to $X'$ that commutes with R-charge and sends $W'$ to $W$.

\begin{conj}\label{birational}
Let $(X,W)$ and $(X', W')$ be birational LG B-models, and assume that $X$ and $X'$ are Calabi-Yau. Then there is a quasi-equivalence
$$Br(X,W) \simeq Br(X', W')$$
\end{conj}

In the next section we prove a special case of this conjecture.

As was explained in the introduction, this is a conservative version of the real conjecture, which is that the B-models associated to $(X,W)$ and $(X',W')$ are equivalent. We state this version since it is not yet proved that the B-model exists.

\pagebreak

\section{Quotients of a vector space by $\C^*$} \label{quotients}

Take a vector space $V$, and equip it with a linear action of $\C^*$, which we'll denote by $\C^*_G$ (the \quotes{gauge group}). We require that $\C^*_G$ acts through $SL(V)$. We have a stack quotient
$$\mathcal{X} = [V / \C^*_G]$$
There are also two possible GIT quotients of $V$ by $\C^*_G$ associated to the characters $\pm 1$ of $\C^*_G$. From the stacky point of view these are open sub-stacks  $$\iota_\pm: X_\pm= [V^{ss}_{\pm} / \C^*_G] \into \mathcal{X}$$ consisting of the semi-stable loci given by either character. All of these spaces are Calabi-Yau.

Now choose an action of $\C^*_R$ on $V$ that commutes with the gauge-group action. Note that both GIT quotients are then preserved by $\C^*_R$.
Let $W$ be a function on $V$ that is invariant with respect to $\C^*_G$ and has R-charge 2. Then we have three Landau-Ginzburg B-models
\begin{equation}\label{3LG}( X_+, \iota^*_+ W)\stackrel{\iota_+}{\into} (\mathcal{X}, W)\stackrel{\iota_-}{\hookleftarrow}
( X_-, \iota^*_-W) \end{equation}

From now on we'll abuse notation and call both $\iota^*_+W$ and $\iota^*_-W$
just $W$.

Both GIT quotients are the total space of orbi-vector bundles
over weighted projective space. To see this, let 
$$V=V_x \oplus V_y \oplus V_z$$
be the decomposition of $V$ into eigenspaces with positive, negative and
zero $\C^*_G$ weights. Then  $X_+$ projects down to $\mathbb{P}V_x$, and
it is the total space of the vector bundle associated to the graded vector
space $V_y\oplus V_z$. Similarly  $X_-$ is the total space of $V_x\oplus V_z$ over $\mathbb{P}V_y$. 

For our sign conventions, it is simplest if we agree that $\P V_y$ is Proj of a \textit{negatively} graded ring, so that the $\cO(-1)$ line bundle on $\P V_y$ is the one that has global sections. If we don't adopt this then whenever we restrict to $X_-$ we have to flip the signs of all line-bundles.

Let $d$ be the sum of the positive eigenvalues of $\C^*_G$ on $V$, 
since $\C^*_G$ acts through $SL(V)$ the sum of the negative eigenvalues is $-d$.

 We'll make repeated use of  the following fairly classical fact:

\begin{lem} \label{wps}\cite{dolgachev} 
\begin{equation*} H^p_{\mathbb{P}V_x}(\cO(k)) = \left\{ \begin{array}{cc} (\cO_{V_x})_k &
p=0,\;\; k\geq 0  \\
 (\cO_{V_x})_{d-k} & p=\mbox{dim} \mathbb{P}V_x, \;\; k\leq -d \\
 0 & \mbox{otherwise}

\end{array}\right.
\end{equation*} 
where $(\cO_{V_x})_k$ is the polynomials on $V_x$ with $\C^*_G$-degree $k$.
\end{lem}

\begin{cor} \label{wpscor}
$$ H^0_{X_+}(\cO(k)) = (\cO_V)_k $$
for all $k$, and 
$$H^p_{X_+}(\cO(k)) = 0$$
for $p>0$ and $k>-d$.
\end{cor}
\begin{proof} By adjunction and affineness of the projection $X_+ \to \mathbb{P}V_x$, we have
$$H^p_{X_+}(\cO(k)) = H^p_{\mathbb{P}V_x}(S^\bullet(V_y\oplus V_z)^\vee\otimes \cO(k))$$
\end{proof}

\subsection{Quasi-equivalences}\label{quasiequivalences}

In this section we will prove
\begin{thm}\label{GITequivalence} There is a natural set of quasi-equivalences
$$Br(X_+, W) \cong Br(X_-, W)$$
parametrised by $\mathbb{Z}$.
\end{thm}

The key idea of the proof of this Theorem comes from \cite{HHP}.
Using restriction functors shown in \ref{3LG}, we will identify both $Br(X_+, W)$ and $Br(X_-, W)$
with one of a set of full subcategories $\mathcal{G}_t \subset Br(\mathcal{X},W)$ parameterized by $t\in \Z$.

 Note that every vector bundle on $\mathcal{X}$ is a direct sum of the obvious
 line bundles  $\mathcal{O}(k), k\in \mathbb{Z}$. Let
$$\mathcal{G}_t \subset Br(\mathcal{X}, W)$$
be the full subcategory consisting of B-branes $(E, d_E)$ where all the summands
of $E$ come from the set
$$\mathcal{O}(t),...,\mathcal{O}(t+d-1)$$
We will show that the functors $$\iota^*_\pm  : Br(\mathcal{X}, W) \to Br(X_\pm,  W)$$
become quasi-equivalences when restricted to any of the subcategories $\mathcal{G}_t$, thus proving Theorem \ref{GITequivalence}.

Recall that a dg-functor between dg-categories is a quasi-equivalence if the induced map on homotopy categories is an equivalence. This means that it must be a quasi-isomorphism on Hom sets (quasi-fully-faithful) and surjective on homotopy-equivalence classes of objects (quasi-essentially-surjective). 

\begin{lem}\label{qff} For any $t\in\Z$, both functors 
$$\iota^*_\pm  : \mathcal{G}_t \to Br(X_\pm,  W)$$
are quasi-fully-faithful.
\end{lem}
\begin{proof}
Obviously we need only show the proof for $\iota^*_+$. Let $(E, d_E)$ and $(F, d_F)$ be any two B-branes in $\mathcal{G}_t$. We get corresponding B-branes $\iota^*_+(E, d_E)$ and $\iota^*_+(F, d_F)$ on $X_+$. Then 
$$\Hom_{Br(\mathcal{X},W)}((E, d_E), (F, d_F)) = R\Gamma_\mathcal{X}(\mathcal{H}om(E,F))$$
and
$$\Hom_{Br(X_+,W)}(\iota_+^*(E, d_E), \iota_+^*(F, d_F)) = R\Gamma_{X_+}(\iota_+^*\mathcal{H}om(E,F))$$
We wish to show that the map $\iota^*_+$ is a quasi-isomorphism between these two complexes.
Recall (Remark \ref{spectralsequence}) that the homology of both complexes can be computed by spectral sequences whose first pages are
$$(H^\bullet_{\mathcal{X}}(\mathcal{H}om(E,F)), d_{E,F}) \;\;\; \mbox{and}\;\;\;(H^\bullet_{X_+}(\iota_+^*\mathcal{H}om(E,F)), \iota^*_+d_{E,F})$$
On $\mathcal{X}$, taking global sections just means taking $\C^*_G$-invariants, which is exact, so for any line-bundle $\cO(k)$,
$$H^\bullet_{\mathcal{X}} (\mathcal{O}(k)) = H^0_{\mathcal{X}}(\cO(k))= (\mathcal{O}_V)_k$$
and by Corollary \ref{wpscor} this is also true on $X_+$ when $k>-d$. Since  $\mathcal{H}om(E,F)$ is a direct sum of line-bundles from the set
$$ \mathcal{O}(1-d),...,\mathcal{O}(d-1)$$
the induced map
$$\iota^*_+: H^\bullet_{\mathcal{X}}(\mathcal{H}om(E,F)) \to H^\bullet_{X_+}(\iota_+^*\mathcal{H}om(E,F))$$ is an isomorphism between the first pages of the two spectral sequences. Hence $\iota^*_+$ is a quasi-isomorphism.
\end{proof}

We will deduce quasi-essential-surjectivity from the following lemma, which
is essentially Beilinson's Theorem \cite{beilinson}.

\begin{lem}\label{resolution} For any $t\in\Z$, any $\C^*_R$-equivariant vector bundle $E$ on $X_+$ has a finite $\C^*_R$-equivariant resolution by direct sums of shifts of line-bundles from the set
$$\mathcal{O}(t),..., \mathcal{O}(t+d-1)$$
\end{lem}
\begin{proof}
Recall that all vector bundles on $\mathcal{X}$ are direct sums of the character line bundles. Since $X_+$ is quasi-projective, $E$ is a quotient of $\iota_+^* V$ for some vector bundle $V$ on $\mathcal{X}$, and we can choose this quotient to be $\C^*_R$-equivariant. Then we have a map $V \to \iota_{+*}E$ which is surjective on $X_+$. Since $\mathcal{X}$ is smooth, the kernel of this map has a finite resolution by vector bundles, which we again may choose to be $\C^*_R$-equivariant. The restriction of this resolution to $X_+$, together with $V$, give a finite $\C^*_R$-equivariant resolution of $E$  by direct sums of character line-bundles. Thus it is sufficient to prove the lemma for the line-bundles $\cO(k)$.

 On $\mathbb{P}V_x$ we have the Euler exact sequence
 $$(\bigwedge^\bullet V_x^\vee, \neg x)\;=  [0\to\cO(-d) \to \; ...\; \to  \cO \to 0]$$
which resolves $\cO(-d)$ in terms of $\cO(-d+1),...,\cO$, and the $\C^*_R$-action on $V_x$ means that it is $\C^*_R$-equivariant. Pull this up to $X_+$. By repeatedly using twists of this exact sequence
we see that any line-bundle $\cO(k)$ has a $\C^*_R$-equivariant resolution
by shifts of line bundles from the set $\cO(t),...,\cO(t+d-1)$.

\end{proof}

\begin{lem}\label{qes} For any $t$, both functors 
$$\iota^*_\pm  : \mathcal{G}_t \to Br(X_\pm,  W)$$
are quasi-essentially-surjective.
\end{lem}
\begin{proof}
Again we only show the proof for $\iota^*_+$. Let $(E, d_E)$ be a B-brane on $(X_+, W)$. By Lemma \ref{resolution} we can  $\C^*_R$-equivariantly resolve $E$ by a complex
$$\mathcal{E}^{-s} \stackrel{\partial_\mathcal{E}}{\to}... 
 \stackrel{\partial_\mathcal{E}}{\to} \mathcal{E}^{-1}
  \stackrel{\partial_\mathcal{E}}{\to} \mathcal{E}^0 \;\; \stackrel{q}{\onto}
  E$$ 
where every term is a direct sum of shifts of line bundles $\cO(k)$ with $t\leq k\leq  t+d-1$. If we let 
$$\mathcal{E}= \bigoplus_p \mathcal{E}^{-p} [p]$$
then $\partial_\mathcal{E}$ is an endomorphism of $\mathcal{E}$ with R-charge 1.  We're going to show that we can perturb $\partial_\mathcal{E}$
 to an endomorphism $d_\mathcal{E}$
whose square is $W\id_\mathcal{E}$, and that the resulting B-brane 
$(\mathcal{E}, d_\mathcal{E})$ is homotopic
to $(E, d_E)$. To see that this proves the lemma, let $\hat{\mathcal{E}}$ be the vector bundle on $\mathcal{X}$ given by the same direct sum of line-bundles as $\mathcal{E}$. Then 
$$H^0_{X_+}(\mathcal{E}nd(\mathcal{E}))
 = H^0_{\mathcal{X}}(\mathcal{E}nd(\hat{\mathcal{E}}))$$
  (see Corollary \ref{wpscor}),
 so $d_\mathcal{E}$ is the restriction of an endomorphism $d_{\hat{\mathcal{E}}}$ of $\hat{\mathcal{E}}$, so we have a B-brane $(\hat{\mathcal{E}}, d_{\hat{\mathcal{E}}})\in \mathcal{G}_t$ that restricts to give $(\mathcal{E}, d_{\mathcal{E}})$. So every B-brane is homotopic to a B-brane lying in $\iota_+^*\mathcal{G}_t$, which is the statement of the lemma.

   As well as the R-charge, we will need to keep track of the grading
on $\mathcal{E}$ that comes from it being a complex, let's call this
the homological grading. Of course $\partial_\mathcal{E}$ also has homological
grade 1.

Now consider the complex $\mathcal{E}$ and the bundle $E$ as objects
in the usual derived category of sheaves on $X_+$, which are quasi-isomorphic
under the map $q$. The line bundles
making up $\mathcal{E}$ have no higher Ext groups between them (Cor. \ref{wpscor} again), so we have quasi-isomorphisms
\begin{equation}\label{qis} H^0( \mathcal{E}nd(\mathcal{E}))
\cong RHom_{X_+}(\mathcal{E}, \mathcal{E})
\cong RHom_{X_+}(E,E)\end{equation}
Here we are using the homological grading on the LHS and the Dolbeaut grading
on the RHS, but the quasi-isomorphims are also equivariant with respect to R-charge. This means we can find an
element $D_0\in H^0(\mathcal{E}nd(\mathcal{E}))$ which is closed
with respect to $\partial_\mathcal{E}$, has R-charge 1, and maps to the endomorphism $d_E$
of $E$, i.e.
$$ d_E q = q D_0$$
We can use $D_0$ to perturb the endomorphism $\partial_\mathcal{E}$ of
$\mathcal{E}$. Unfortunately this does not yet make it a B-brane
for $(X_+, W)$, rather we have
$$(\partial_\mathcal{E}+D_0)^2 = D_0^2 = W \id_{\mathcal{E}} - [ \partial_{\mathcal{E}}, D_{-1}]$$
for some element $D_{-1}\in H^0(\mathcal{E}nd(\mathcal{E}))$ which
has homological grade -1 and R-charge 1. Here we write $[\partial_\mathcal{E},
-]$ to denote the supercommutator with respect to the R-charge grading, strictly
speaking this is the differential on $H^0(\mathcal{E}nd(\mathcal{E}))$ that
comes from considering $(\mathcal{E},0)$ as a B-brane on $(X_+,0)$ rather
than as a complex of sheaves in $D^b(X_+)$, but the difference is irrelevant
and the signs are more convenient this way.

If we perturb further by $D_{-1}$ we get
$$(\partial_\mathcal{E}+D_0+D_{-1})^2 = W\id_\mathcal{E} + [D_0, D_{-1}]
+ D_{-1}^2$$
and notice that now all the unwanted terms have homological degree at most
-1. We claim we can iterate this process, and since the homological degree
is bounded it will terminate. Indeed, we wish to solve
$$(\partial_\mathcal{E} + D)^2 = W\id_\mathcal{E}$$
where 
$$D = D_0 + D_{-1} + D_{-2}+ ...$$
is a series of terms of decreasing homological grade and R-charge 1. The
piece of this equation in homological grade $-k<0$ is
$$[\partial_\mathcal{E}, D_{-k-1}] + (D^2)_{-k} = 0 $$
Assume that we have found $D_0,.., D_{-k}$ such that this equation holds in homological grades $>-k$. By (\ref{qis}), $H^0(\mathcal{E}nd(\mathcal{E}))$ has no homology in negative
degrees, so we can find $D_{-k-1}$ if $(D^2)_{-k}$ is closed. But
\begin{eqnarray*} [\partial_\mathcal{E}, (D^2)_{-k}]& = &[\partial_\mathcal{E}, D^2]_{-k+1}\\
&=& \sum_{i=0}^{-k+1}[[\partial_\mathcal{E}, D_{i-1}],D_{-k-i+1}]\\
&=& \sum_{i=0}^{-k+1}[(D^2)_i,D_{-k-i+1}]\\
& =& [D^2, D]_{-k+1}\\
&=&0
\end{eqnarray*}
so inducting on $k$ our solution $D$ exists. We let
$$d_\mathcal{E} = \partial_\mathcal{E} + D$$
so $(\mathcal{E}, d_\mathcal{E})$ is a B-brane on $(X_+, W)$. It remains
to show that it is homotopic to the brane $(E, d_E)$. To see this we consider
the dga
$$\End_{Br(X_+, W)}((E,d_E)\oplus (\mathcal{E},d_\mathcal{E}))= 
\Gamma(\mathcal{E}nd(E\oplus\mathcal{E})\otimes \mathcal{A}^{0,\bullet})$$
This carries its usual grading (the sum of R-charge and Dolbeaut grade) and also the homological
grading from $\mathcal{E}$. Its differential is a sum of terms induced from
$d_E, \bar{\partial}, \partial_\mathcal{E}$ and the $D_{-k}$, these have
homological grading $0, 0, 1$ and $-k$ respectively. Thus we can filter this
dga by defining 
$$F^p\End_{Br(X_+, W)}((E,d_E)\oplus (\mathcal{E},d_\mathcal{E}))$$
to be the sum of the bi-graded pieces that have
$$(\mbox{usual grade}) - (\mbox{homological grade}) \geq p$$
then this filtration is compatible with the differential
and the algbra structure. Also the filtration
is bounded, in the sense that the induced filtration on any (usual) graded
subspace is bounded. This is a sufficient condition for the associated spectral
sequence of dgas to converge \cite{mccleary}. To get page 1 of this spectral
sequence we take the homology of the term of the differential which has bi-degree
$(1,1)$, this is the term induced from $\partial_\mathcal{E}$. The diffential
on page 1 is induced from $d_E, \bar{\partial}$ and $D_0$, and $D_0$ was
chosen so that it induced $d_E$ on $\partial_\mathcal{E}$-homology. So page 1 is 
$$\Gamma(\mathcal{E}nd(E\oplus E)\otimes \mathcal{A}^{0,\bullet}) 
= \End_{Br(X_+, W)}((E,d_E)\oplus (E,d_E)) $$
This is concentrated in homological grade zero, so the spectral sequence
collapses at page 2. We deduce that in the homotopy category the objects
$(\mathcal{E}, d_\mathcal{E})$ and $(E, d_E)$ are isomorphic.
\end{proof}

\subsection{Spherical B-branes}\label{sphericaltwists}

We use the same set-up as in the previous subsection, but from now on we
assume that $\C^*_G$ has no zero eigenvalues in $V$, so
$$V = V_x \oplus V_y$$
are the positive and negative $\C^*_G$-eigenspaces.

The zero section gives an inclusion
$$\mathbb{P} V_x \into X_+$$
and there is an associated sky-scraper sheaf $\cO_{\mathbb{P}V_x}$. This is a spherical object in the derived category $D^b(X_+)$. We are going to modify it so as to produce a spherical object in the category of B-branes $Br(X_+, W)$.

Under our definition a B-brane is a vector bundle, so it is supported over the whole of $X_+$ (it is \quotes{space-filling}). However a better definition should allow arbitrary coherent sheaves, which in particular can be supported just on subschemes. Then no modification of $\cO_{\mathbb{P}V_x}$ would be necessary, we could just equip it with the zero endomorphism, which does indeed square to $W$ because $W \equiv0$ along the zero section. 

We have not attempted to develop such a definition because the presence of local Ext groups makes defining the morphisms between such objects significantly more difficult. Instead we shall resolve $\cO_{\mathbb{P}V_x}$ by vector bundles, and deform the resolution. Nevertheless the resulting object does behave as if it was supported just on the zero section (Prop. \ref{proponhoms}).

Let $\set{\partial_{y_i}},\set{dy^i}$ be dual bases of $V_y$ and $V_y^\vee$, and $y_i$ the corresponding co-ordinates.
Consider the Koszul resolution of $\cO_{\mathbb{P}V_x}$:
$$ (\bigwedge^\bullet V_y^\vee,\neg \Sigma_i y_i \partial_{y_i}) \iso \cO_{\mathbb{P}V_x}$$
We will deform the differential to make it a B-brane on $(X_+, W)$, and show that it is still spherical.

Write $W$ as 
\begin{equation}\label{splitW} W = \sum_i y_i f_i\end{equation}
This is possible since $W$ is gauge invariant, and has R-charge 2 so has
no constant term. We define a
B-brane on $(X_+, W)$ by the $\C^*_R$-equivariant vector bundle
$$ S := \bigwedge^\bullet( V_y^\vee[1])$$
and the endomorphism
$$d_S:=\sum_i \neg y_i\partial_{y_i} + \wedge f_i dy^i$$
It is easy to check that $d_S^2 = W\id_S$.

\begin{prop} The B-brane $(S, d_S)$ is independent, up to isomorphism, of
the choice of splitting (\ref{splitW})\end{prop}
\begin{proof}
Let $W = \sum_i y_i \hat{f}_i $ be another choice of splitting, and $\hat{d}_S$
the corresponding endmorphism of $S$. It is sufficient to prove the lemma in the case that $f_i=\hat{f}_i$ for $i>2$. In that case
we have 
$$\hat{f}_1 = f_1 + y_2 g$$
$$\hat{f}_2 = f_2 - y_1 g$$
for some $g$. We have inverse isomorphisms
$$ \id_S + \wedge (g dy^1\wedge dy^2) : (S, d_S) \to (S, \hat{d}_S)$$
$$ \id_S - \wedge (g dy^1\wedge dy^2) : (S, \hat{d}_S) \to (S, d_S)$$
and it is easy to check that these are closed.
\end{proof}

Let $\zeta: \mathbb{P}V_x \into X_+$ denote the zero section.

\begin{prop}\label{proponhoms} For any B-brane $(E,d_E)$ on $(X_+,W)$, the homology of
$$\Hom_{Br(X_+,W)}((E, d_E), (S, d_S))$$
can be computed from a spectral sequence whose first page is
$$H^\bullet_{\mathbb{P}V_x} (\zeta^*E^\vee)$$
 with the differential induced from $d_E$.
\end{prop}
Note that since $W=0$ on the zero section, $d_E$ does indeed induce a differential on $H^\bullet_{\mathbb{P}V_x} (\zeta^*E^\vee)$.
\begin{proof}
The bundle $S$, as well as being $\C^*_R$-equivariant, is graded by the powers
in the exterior algebra. Let's call this the exterior grading, and write
$$d_S = \partial_S +D_S$$
for the terms of exterior grade -1 and +1 ($\partial_S$ is the usual
Koszul differential). Consider
$$\Hom_{Br(X_+, W)}((E, d_E), (S, d_S)) = \Gamma(\mathcal{H}om(E,S)\otimes \mathcal{A}^{0,\bullet})$$
This carries its usual grading which is the sum of R-charge and the Dolbeaut grading, and also an exterior grading induced from the grading on $S$. The
differential has terms induced from $\partial_S$, $D_S$, $d_E$ and $\bar{\partial}$ having bi-degrees
$(1,-1)$, $(1,1)$, $(1,0)$ and $(1,0)$ respectively. 
We now proceed by a similar argument to the one used at the end of Lemma
\ref{qes}. Define a filtration by
letting
$$F^p\Hom_{Br(X_+, W)}((E, d_E), (S, d_S)) \subset \Hom_{Br(X_+, W)}((E, d_E), (S, d_S))$$
be the direct sum of the bi-graded pieces whose total degree is $\geq p$,
then the differential preserves this filtration, and is bounded for any fixed
total of the Dolbeaut grade and R-charge. Page 1 of the associated spectral
sequence is given by taking the homology of the term induced from $\partial_S$ only, so
it is
$$\Gamma(\mathcal{H}om(E, 
    \cO_{\mathbb{P}V_x})\otimes \mathcal{A}^{0,\bullet})
 \cong R\Gamma_{\mathbb{P}V_x}(\zeta^*E^\vee)$$ 
with differential induced from $d_E$ and $\bar{\partial}$. This is concentrated in exterior grade zero, so this spectral sequence collapses after this page.

To compute page 2, we can use a second spectral sequence (essentially the one from Remark \ref{spectralsequence}) by remembering that the complex on page 1 is actually a bi-complex under the Dolbeaut grading and R-charge.
\end{proof}

\begin{cor} \label{sphericalobject}$(S, d_S)$ is either a spherical object or zero in $H_0(Br(X_+, W))$.
\end{cor}
\begin{proof}
 By Corollary \ref{wpscor}
\begin{eqnarray*}H^\bullet_{\mathbb{P}V_x}(\zeta^*S^\vee)& =& H^\bullet_{\mathbb{P}V_x}(\cO) \oplus H^\bullet_{\mathbb{P}V_x}(\cO(-d))\\
&=& \C \oplus \C
\end{eqnarray*}
where the second copy of $\C$ has some bi-degree depending on the dimensions and R-charges of $V_x$ and $V_y$. Either the spectral sequence collapses at this point (which it
usually will for degree reasons) and $(S, d_S)$ is spherical, or it converges
to $0$ and $(S, d_S)$ is contractible.
\end{proof}

\begin{eg} \textit{(Flop with superpotential)} Let $V=\C^4$ with $\C^*_G$
weights 1,1,-1,-1, so both GIT quotients are isomorphic to $\cO(-1)_{\mathbb{P}^1}^{\oplus
2}$. Let $W = x_1 y_1 + x_2 y_2$ (and pick any compatible $\C^*_R$ action).
We can take $(S,d_S)$ to be 
$$\xymatrix@=35pt{ \cO(2) \ar@<1ex>[r]^(.4){(y_2, -y_1)} & \cO(1)^{\oplus 2} \ar[l]^(.6){(x_2, -x_1)} \ar@<1ex>[r]^(.6){(y_1, y_2)} & \cO \ar[l]^(.4){(x_1, x_2)}
   }$$
so
$$\Hom_{Br(X_+,W)} ((S, d_S), (S, d_S)) \cong R\Gamma_{\mathbb{P}^1}(\zeta^*S^\vee)
\cong 0$$  
and so $(S,d_S)$ is contractible. In fact one would expect the whole category $Br(X_+,W)$ in this example to be zero by Kn\"orrer
periodicity.
\end{eg}

\subsection {Spherical twists}

We continue with the same class of examples as in the previous subsection. We have shown in Theorem \ref{GITequivalence} that for each $t\in \Z$ we have quasi-equivalences
$$ Br(X_+, W) \stackrel{\iota_+^*}{\longleftarrow} \mathcal{G}_t \stackrel{\iota_-^*}{\longrightarrow} Br(X_-, W)$$
On the homotopy categories these can be inverted, so we have $\Z$-many equivalences
$$\Phi_t: H_0(Br(X_+,W)) \iso H_0(Br(X_-,W))$$
passing through the categories $H_0(\mathcal{G}_t)$, and hence we have autoequivalences $\Phi_{t+1}^{-1}\Phi_t$ of $Br(X_+,W)$. The statement that we would like to be able to make is that $\Phi_{t+1}^{-1}\Phi_t$ is an inverse spherical twist around the spherical object $(S(t), d_S)$, in the sense of \cite{ST}. Unfortunately such a statement would require a proper theory of Fourier-Mukai transforms for Landau-Ginzburg B-models, and we have not developed such a theory. Instead we're going to settle for a less clean statement, which we prove below (Theorem \ref{sphericaltwistthm}). 

Recall that an inverse spherical twist on a space $X$ is an auto-equivalence of the derived category $D^b(X)$ that sends an object $\mathcal{E}$ to the cone on the natural map
$$[\mathcal{E} \longrightarrow \RHom_X(\mathcal{E}, S)^\vee \otimes S]$$
where $S$ is a fixed spherical object in $D^b(X)$. We have shown (Cor. \ref{sphericalobject}) that we have an object $(S, d_S)\in Br(X_+,W)$ that is either spherical or zero, we can twist it by $\cO(t)$ to get other B-branes $S(t)$ that are either spherical or zero. What we're going to do is construct, for any B-brane $(E,d_E)\in Br(X_+,W)$, a suitable map
$$\epsilon_E: E \to \mathcal{H}^\vee\otimes S(t)$$
where $\mathcal{H}$ is a complex such that
$$\mathcal{H}\simeq \Hom_{Br(X_+,W)}(E,S(t))$$
and then show that $\Phi_{t+1}^{-1}\Phi_t$ sends $E$ to the cone on $\epsilon_E$. If $S(t)$ is spherical, this is a good approximation to showing that $\Phi_{t+1}^{-1}\Phi_t$ is a spherical twist (at least on objects). If $S(t)$ is zero, it shows that $\Phi_{t+1}^{-1}\Phi_t$ is the identity (at least on objects).

We begin with another Corollary of Proposition \ref{proponhoms}.

\begin{lem}\label{supponP} Let $(E, d_E)\in \iota^*_+\mathcal{G}_t$. Then 
$$\Hom_{Br(X_+,W)}((E, d_E), (S(t), d_S))\cong (H^0_{\mathbb{P}V_x} (\zeta^*E^\vee(t),\;
d_E^\vee)$$
\end{lem}
\begin{proof}
$$\Hom_{Br(X_+,W)}(E, S(t))= \Hom_{Br(X_+,W)}(E(-t), S)$$
which by Prop. \ref{proponhoms} can be computed from $H^\bullet_{\mathbb{P}V_x} (\zeta^*E^\vee(t))$.
But $E$ is a direct sum of line bundles $\cO(k)$ with $t\leq k <t+d$, so by Lemma
\ref{wps},  
$$H^\bullet_{\mathbb{P}V_x} (\zeta^*E^\vee(t)) = H^0_{\mathbb{P}V_x} (\zeta^*E^\vee(t)) = \C^{\oplus m_E}$$
 where $m_E$ is the number of copies of $\cO(t)$ appearing
in $E$, and the spectral sequence collapses.
\end{proof}

Pick an $(E,d_E)\in \iota^*_+ \mathcal{G}_t$. For notational convenience let us define
$$\mathcal{H}:=(H^0_{\mathbb{P}V_x} (\zeta^*E^\vee(t)),\;
d_E^\vee)$$
If we were in the special case when $W=0$ and we had chosen $d_E=0$ then there would be a canonical map (the unit of the adjunction)
$$\epsilon_0: E \to \mathcal{H}^\vee\otimes S(t)$$
This map just projects $E$ onto its $\cO(t)^{\oplus m_E}$ summand and then includes this as the final term of $\mathcal{H}^\vee\otimes S(t)$. 

When $d_E \neq 0$ the map $\epsilon_0$ is not closed, so we cannot take its mapping cone. We can fudge this using the following:

\begin{lem}
There is a closed map of R-charge 0
$$\epsilon_E = \epsilon_0 + \epsilon_{1} + ... : E \to \mathcal{H}^\vee\otimes S(t)$$
where $\epsilon_i$ has exterior grade $i$. 
\end{lem}
Recall that the \quotes{exterior grade} refers to the grading on $S$ that comes from its underlying vector bundle being an exterior algebra.
\begin{proof}
We use the iterative technique from Lemma \ref{qes}. Consider the complex
$$\Hom_{X_+}(E, \mathcal{H}^\vee\otimes S(t))$$
This is bigraded by R-charge and exterior grade, and carries a differential $d$ composed of terms
$$d= d_{-1} + d_0 + d_{1}$$
of exterior grade -1, 0, and 1. The term $d_{-1}$ just comes from the Koszul differential $\partial_S$ on $S$. If we just take $d_{-1}$ homology, the complex is acyclic except in exterior grade 0, where it is
$$\Hom_{\P V_x}(\cO(t)^{\oplus m_E},\cO(t)^{\oplus m_E})$$
 We want to solve $d \epsilon_E = 0$, which in exterior grade $k$ is
$$ d_{-1} \epsilon_{k+1} = -d_0\epsilon_k- d_1\epsilon_{k-1} $$
Suppose we have solved this for all exterior grades $\leq k$. Then
\begin{eqnarray*}d_{-1}(-d_0\epsilon_k - d_1\epsilon_{k-1})& = &d_0d_{-1}\epsilon_k + d_1d_{-1}\epsilon_{k-1} + d_0^2 \epsilon_{k-1}\\
&=&-d_0(d_0\epsilon_{k-1} + d_1\epsilon_{k-2}) - d_1(d_0\epsilon_{k-2}+d_1\epsilon_{k-3}) +d_0^2\epsilon_{k-1} \\
&=&0
\end{eqnarray*} 
If $k\geq 1$ then by acyclicity an $\epsilon_{k+1}$ exists. To check that an $\epsilon_1$ exists we need to check that $d_0\epsilon_0$ is zero in $d_{-1}$-homology, which means calculating the component of it that maps $\cO(t)^{\oplus m_E}\subset E$ to $\cO(t)^{\oplus m_E}\subset \mathcal{H}^\vee\otimes S(t)$. But this is zero, because the differential on $\mathcal{H}^\vee$ cancels the component of $d_E$ that maps $\cO(t)^{\oplus m_E}$ to itself.
\end{proof}

Write $(C_E, d_C)$ for the mapping cone of $\epsilon_E$.

\begin{thm} \label{sphericaltwistthm}For any $(E, d_E)\in \iota^*_+\mathcal{G}_t$,
$$\Phi_{t+1}^{-1}\circ \Phi_t([(E,d_E)]) \simeq [(C_E,d_C)]$$
in the homotopy category of $Br(X_+,W)$.
\end{thm}
\begin{proof}
Calculating $\Phi_t([(E, d_E)]$ is easy since $(E, d_E)\in \iota^*_+\mathcal{G}_t$,
it is given by exactly the same data as $(E,d_E)$ but considered as a brane on $X_-$.
To apply $\Phi_{t+1}^{-1}$ to it we have to replace it with a homotopy equivalent
brane that lies in $\iota^*_- \mathcal{G}_{t+1}$, which we know we can do by Lemma
\ref{qes}. In fact we can do this fairly explicitly: split $E$ into its factors
$$E = \cO(t)^{\oplus m_E}\oplus E'$$
where $E'$ is a direct sum of line bundles from $\set{ \cO(t+1),...,\cO(t+d-1)}$,
then we can resolve $E$ (recall Lemma \ref{resolution}) by the complex
$$(\mathcal{E}, \partial_{\mathcal{E}}):=\bar{S}(t)^{\oplus m_E} \oplus E'  \iso E$$
where $\bar{S}$ is the complex
$$(\bar{S}, \partial_{\bar{S}}) := (\bigwedge^{\geq 1} (V_y^\vee[1]),  \neg \Sigma_i y_i\partial_{y_i})$$
given by truncating $(S, \partial_S)$. Now we run the algorithm of Lemma \ref{qes} to get a brane $(\mathcal{E},
d_{\mathcal{E}})\in \iota^*_-\mathcal{G}_{t+1}$, which we can then transport
back to $X_+$. This brane is graded by the powers of the exterior algebra, as before we call this the exterior grading.

 Define an exterior grading on $C_E$ by putting $E$ in grade zero and shifting the exterior grading on $\mathcal{H}^\vee\otimes S(t)$ by 1 (as one usually would for a mapping cone). The differential on $C_E$ is then a sum of terms of exterior grade $\geq -1$, and the term of exterior grade -1 is just the term induced from the Koszul differential $\partial_S$ on $S$. Denote this term by $\partial_C$. Then $(\mathcal{E}, \partial_\mathcal{E})$ and $(C_E, \partial_C)$ are branes on the LG model $(X_+, 0)$, and they are clearly homotopy equivalent. Indeed, $(C_E, \partial_C)$ is
$$[\cO(t)^{\oplus m_E} \oplus E' \stackrel{(j^{\oplus m_E},0)}{\longrightarrow} ( S(t)^{\oplus m_E} ,\partial_S)]
   $$

where $j: \cO \to S$ is the inclusion of $\cO = \bigwedge^0 V_y^\vee \into S$,
and the cone on $j$ is clearly homotopy equivalent to $(\bar{S}, \partial_{\bar{S}})$. This means we have maps 
$$ \xymatrix{ 
C_E \ar@(ul,dl)[]_{h_1} \ar@/^/[rr]^{f_0} 
&& \mathcal{E} \ar@/^/[ll]^{g_0}  \ar@(dr,ur)[]_{\hat{h}_1}
} $$
forming a homotopy equivalence (with respect to $\partial_C$ and $\partial_{\mathcal{E}}$), where $f_0$ and $g_0$ have both R-charge and exterior grade 0 and $h_1$ and $\hat{h}_1$ have R-charge -1 and exterior grade 1. We claim we can use our iterative trick once again to perturb these maps by terms of increasing exterior grade until we get a homotopy equivalence between $(C_E, d_C)$ and $(\mathcal{E}, d_{\mathcal{E}})$. The argument is much the same as before: firstly observe that 
$$RHom_{X_+}((C_E\oplus \mathcal{E}, \partial_C\oplus \partial_{\mathcal{E}}),(C_E\oplus \mathcal{E}, \partial_C\oplus \partial_{\mathcal{E}}))$$
has homology only in exterior grade zero, because $\partial_C$ and $\partial_{\mathcal{E}}$ have homology only in exterior grade zero. Secondly, let
$$F_0 = \left( \begin{array}{cc} 0 & f_0 \\ g_0 & 0 \end{array} \right) \;\;\;\;\; H_1 = \left( \begin{array}{cc} h_1 & 0 \\ 0 &\hat{h}_1 \end{array} \right)$$
be the elements of this dga that we want to perturb, and let $d=d_C\oplus d_{\mathcal{E}}$ and $\partial = \partial_C\oplus \partial_{\mathcal{E}}$. The equations we want to solve are
$$[ d, F] = 0$$
$$F^2 = \id + [d, H]$$
which are equivalent to
$$[\partial, F] = -[(d-\partial), F]$$
$$[\partial, H] = F^2-\id - [(d-\partial), H]$$
and it is easy to check that if these equations hold in exterior grade $\leq k$ then the right-hand-sides are closed with respect to $[\partial,-]$ so by acyclicity they can be solved in exterior grade $k+1$.
\end{proof}

% ----------------------------------------------------------------

\bibliographystyle{plain}
\bibliography{mybib}

\end{document}